%% file: main.tex
\documentclass[12pt,titlepage]{article}
\usepackage{rsipacks}

\usepackage{fancyhdr}
\input{preamble}

\begin{document}

\include{cover}

\begin{singlespace}
\maketitle
\end{singlespace}

\begin{abstract}
\input{abstract}

\vspace{1in}
\begin{center}\textbf{Summary}\end{center}
\input{summary}
\end{abstract}

\include{paper}

\begin{singlespace}
\include{biblio}

\end{singlespace}

\include{appa}
\end{document}

%% file: preamble.tex
\usepackage{graphicx}
\usepackage{amsmath, amsthm}
\usepackage{lastpage}

\newtheorem{theorem}{Theorem}
\newtheorem{lemma}{Lemma}
\theoremstyle{definition}
\newtheorem{definition}{Definition}

\newcommand{\Ushape}{$\mathbf{U}$}

%% file: cover.tex
% This is cover.tex

% Insert your information as appropriate.

\title{
Complexity of Interlocking Polyominoes
}

\author{
Sidharth Dhawan\\
Massachussets Institute of Technology \\
sidharth@mit.edu\\
\vspace{0.5in}\\
%under the direction of\\
Zachary Abel\\
Massachusetts Institute of Technology\\
zabel@mit.edu\\
\vspace{1in}
}

% Don't change the information below.

\date{
July 2011
}

%% file: abstract.tex
% This is abstract.tex

Polyominoes are a subset of polygons which can be constructed from integer-length squares fused at their edges. A system of polygons $P$ is interlocked if no subset of the polygons in $P$ can be removed arbitrarily far away from the rest. It is already known that polyominoes with four or fewer squares cannot interlock. It is also known that determining the interlockedness of polyominoes with an arbitrary number of squares is PSPACE hard. Here, we prove that a system of polyominoes with five or fewer squares cannot interlock, and that determining interlockedness of a system of polyominoes including hexominoes (polyominoes with six squares) or larger polyominoes is PSPACE hard.

%% file: summary.tex
% This is summary.tex

The goal of this research was to discover which sets of grid-aligned polyominoes could be collectively interlocked. A system of shapes is collectively interlocked if none can move arbitrarily far away from the rest by any kind of rotation or translation. For example, two links in a chain are collectively interlocked in three dimensions, but two knotted strings are not, because a knot can be untied by a specific set of movements. A polyomino is a shape made entirely of $n$ unit squares fused at the edges. For example, the game ``Tetris" is played with tetrominoes, or shapes made from four fused unit squares. The study of collective polyomino interlockedness may have applications in the fields of protein folding, motion planning (which is an important problem in artificial intelligence), and material science. In this paper, we prove that a system of polyominoes made of five or fewer squares can never be collectively interlocked. We further prove that determining interlockedness of a system in which polyominoes with six or more squares are present is computationally intractable.

%% file: paper.tex
% This is paper.tex

\section{Introduction}

\subsection{Polygons and Polyominoes}
A polyomino is a shape assembled entirely from squares attached at edges. Here, we explore the notion of {\it collective interlockedness} of a system of polyominoes. A system of polyominoes is collectively interlocked if no polyominoes can be moved arbitrarily far away from any of the others. It is not interlocked if any subset of polyominoes can move arbitrarily far away from the rest. We only consider systems of polyominoes whose squares are initially placed in grid-aligned configurations, but we allow the polyominoes to rotate at any angle from their original configurations. We discuss the computational complexity of deciding different instances of polyomino interlockedness.

We start in Section~\ref{sec:terms} by defining terms that will be used in later theorems. In section 3 we attempt to create an upper bound for non-interlockedness. It is known that polyominoes with four or fewer squares (as well as the majority of pentominoes) belong to a class of polyominoes that cannot interlock. We show by analysis of special cases that pentominoes also cannot interlock. We also consider lower bounds for interlockedness in section 4. We show a novel construction of interlocking hexominoes and prove that it is rigid (i.e. that no polyomino can be translated at all). In section 5 we consider a lower bound for PSPACE hardness of polyomino interlockedness. We examine Demaine and Hearn's proof of PSPACE hardness of polyomino interlockedness and show that this proof holds for polyominoes with 6 or more squares.

\section{Background} \label{sec:terms}

First, we formally define polyominoes, interlockedness, and PSPACE hardness, which are important basic concepts in this paper.

\begin{definition}
{\it A polyomino is a polygon with integer side lengths and only right angles that can be created by joining squares of unit length at the edges.}
\end{definition}

\begin{definition}
{\it A system of polygons $P$ is} interlocked {\it if and only if}
\begin{itemize}
\item {\it it is arranged in a grid-aligned configuration}
\item {\it no subset of these polygons can be separated arbitrarily far from the rest by rotation at any angle or translation. In particular, we allow simultaneous movement of shapes.}
\item {\it polygons never intersect at any time.}
\end{itemize}
\end{definition}

\begin{definition}
{\it A problem belongs to the class} PSPACE {\it if a computer with finite memory can find its solution by writing to and reading from B bits, where B is a polynomial function of the size of the inputs.}
\end{definition}
We often define complexity classes based on the amount of time (number of steps) or space (amount of memory) required by a computer to solve problems in these classes. For example, problems in the class P can be solved by computers in polynomial time.
A problem is hard for the class PSPACE if it is at least as difficult as any problem in the class PSPACE. Problems hard for the class PSPACE are generally considered computationally intractable.

We now define rigidity, which is a stronger version of interlockedness. Later in this paper, we will be proving rigidity of certain systems of polyominoes, because it is in some cases easier to prove.

\begin{definition}
{\it A system of polygons is} rigid {\it if, given any configuration in which all polyominoes except one have been displaced by at most $\epsilon$, no polyomino has been displaced (for small enough $\epsilon$).}
\end{definition}

Claiming that a system of polygons is rigid is stronger than claiming they are not interlocked, because in a rigid system no polygon can be displaced at all, whereas interlocked polygons can be displaced a finite distance relative to each other (though not an infinite distance). 

We now turn our attention to special classes of polyominoes. These classes will be mentioned in the proofs, since some classes can be spread apart more easily.

%\begin{definition}
%{\it An} orthogonal {\it polygon has sides which are parallel to either the $x$ or $y$ axis.}  
%\end{definition}
%
%Orthogonal polygons are similar to polyominoes, but are not restricted to having integer side lengths.

	We define a class of polygons that can be easily separated in just one direction. A polygon is {\bf monotone} in direction $\theta$ if the region of intersection of the polygon with any line drawn perpendicular to $\theta$ is no more than one connected line segment, as shown in Figure~\ref{ymonotone}. Some polyominoes are also monotone in both the $x$ and $y$ directions, which makes them even easier to separate from other polyominoes. An orthogonal polygon is {\bf orthogonally convex} if it is monotonic in both the $x$ and $y$ directions.

\begin{figure}[h]
\begin{center}
\includegraphics[height=0.5in]{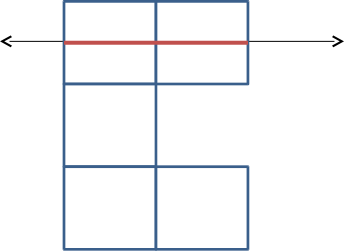}
\end{center}
\caption{A $y$-monotone polygon}
\label{ymonotone}
\end{figure}

\begin{definition}
{\it The} pockets {\it in direction $\theta$ of a polygon $P$ are the largest disjoint regions not present in $P$ that are present in the smallest polygonal region that contains $P$ and is monotone in direction $\theta$.}
\end{definition}
A pocket is shown in Figure~\ref{pocket}.

\begin{figure}[h]
\begin{center}
\includegraphics[height=0.5in]{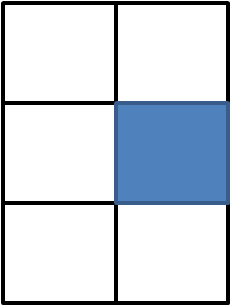}
\end{center}
\caption{A pocket}
\label{pocket}
\end{figure}

\begin{definition}
{\it A} translational ordering (TO) {\it is a non-simultaneous sequence of translations that allows the polygons to be moved arbitrarily far away from each other on the plane. }
\end{definition}

\begin{definition}
{\it A} unidirectional translational ordering (UTO) {\it in a direction $\theta$ is a translational ordering in which the polygons are restricted to moving in only a single direction, $\theta$.}
\end{definition}	
	Claiming that a TO exists for a system of polygons is stronger than saying they are not interlocked, because simultaneous motion is not allowed in a TO. These definitions will also be used when discussing different cases of separability.

\section{Non-Interlockedness of Pentominoes}
Our first consideration is finding a lower bound in the number of squares for polyomino interlockedness. Haldar \cite{Haldar} showed that a system of orthogonally convex polygons can never be interlocked. Since all polyominos made of four squares or fewer are orthogonally convex, it follows that a UTO in both the $x$ and $y$ directions exists for all tetrominoes, triominoes, and dominoes. We cannot make the same claim for pentominoes, because not all pentominoes are orthogonally convex. Here we prove that pentominoes still cannot interlock.
We begin with two important lemmas from \cite{Haldar}.

\begin{theorem}
A system of grid-aligned polyominoes with five or fewer squares can never be interlocked.
\end{theorem}

\begin{lemma}\label{lemma:monotone}
A UTO of a system of polygons exists in direction $\theta$ if all the polygons are monotone in direction $\theta + \frac{\pi}{2}$ \cite{Haldar}. 
\end{lemma}

%\begin{figure}[ht]
%\begin{center}
%\includegraphics[height=0.5in]{monotone_UTO.png}
%\end{center}
%\caption{UTO for a system of monotone polygons}
%\end{figure}

\begin{lemma}\label{lemma:convex}
 A UTO in both the $x$ and $y$ directions exists for any system of orthogonally convex polygons \cite{Haldar}.
\end{lemma}

\begin{proof}
As shown in Figure~\ref{pents}, all the pentominoes except the \Ushape{} (highlighted in Figure~\ref{pents}) are orthogonally convex. The \Ushape{} is monotone in either the $y$ direction or the $x$ direction.

\begin{figure}[h] \label{pents}
\begin{center}
\includegraphics[height=1in]{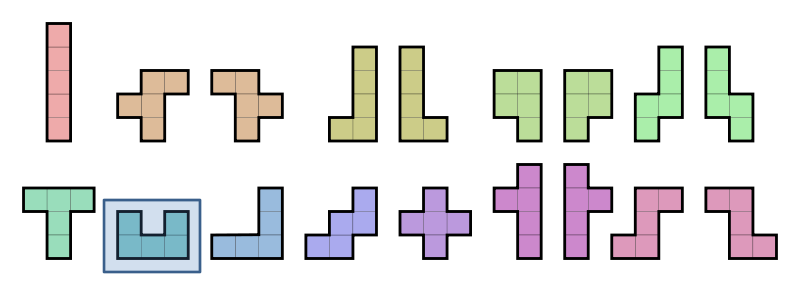}
\caption{All pentominoes, with the \Ushape{} highlighted. Taken from \cite{Wikipedia}}

\end{center}
\end{figure}

To prove that this subset of polyominoes can not interlock, we present an algorithm to spread them apart. Roughly, we will group our original polyominoes into large collective shapes which can be spread apart from each other more easily, after which all the polyominoes inside the groups can be spread apart, until no more groups are left. Specifically, we define each $y$-monotone \Ushape{} as its own group. Next, we group together any $x$-monotone \Ushape{} with the polyomino touching its pocket. We define any polyomino that is not in a group as its own group. We claim that all the resulting groups are $y$-monotone, and thus, they can all be separated by Lemma~\ref{lemma:monotone}. This claim is proved in the next paragraph. Since the only pentomino that is not $x$-monotone is the $y$-monotone \Ushape, we claim that each of the polyominoes in groups consisting of multiple shapes must be $x$-monotone. Thus, these groups can all be separated in the $y$ direction by Lemma~\ref{lemma:monotone}.

\begin{figure}[h]
\begin{center}
\includegraphics[height=1in]{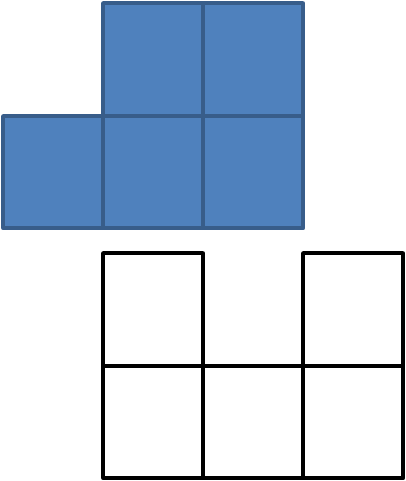}
\end{center}
\caption{case 1: pocket is not filled}
\label{ucase1}
\end{figure}

We now prove the claim that our groups are $y$-monotone. We know that all $y$-monotone \Ushape 's and all other polyominoes with five or fewer squares are $y$-monotone. We examine four cases in which $x$-monotone \Ushape 's are grouped into groups of more than one polyomino. In the first case, as shown in Figure~\ref{ucase1}, there is no shape touching the \Ushape 's pocket. In this case, we may treat the \Ushape{} as a $2 \times 3$ rectangle, which is $y$-monotone. In the second case, there are two polyominoes in the group, and it is orthogonally convex. In this case, the group is by definition $y$-monotone.

\begin{figure}[h]
\begin{center}
\includegraphics[height=1in]{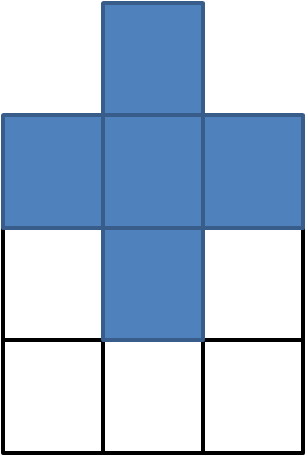}
\end{center}
\caption{case 2: the group is collectively orthogonally convex}
\end{figure}

In the third case, there are two polyominoes in the group, which is not orthogonally convex. It can be shown by a case analysis (see Figure~\ref{ucase2}) that all of these groups are $y$-monotone, although they are not $x$-monotone. 

\begin{figure}[h]
\begin{center}
\includegraphics[height=1in]{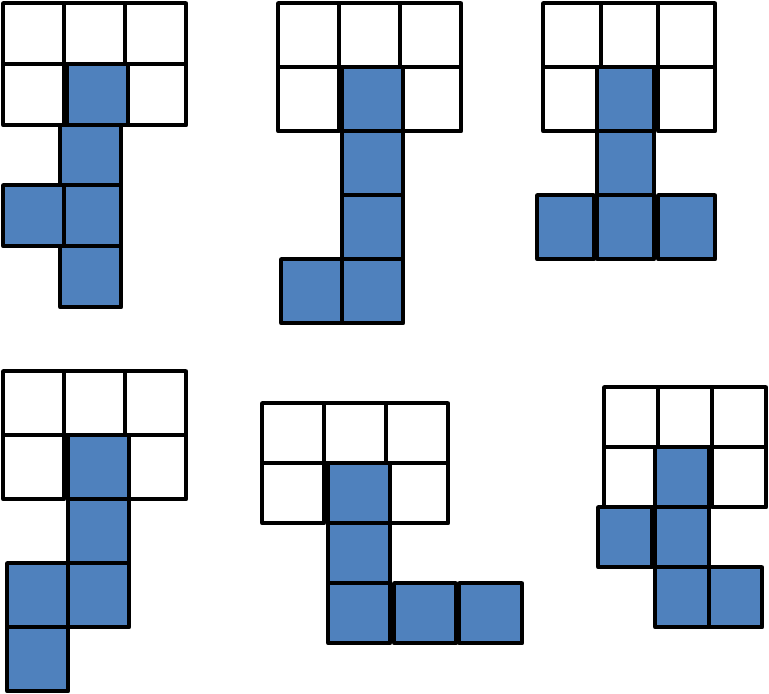}
\end{center}
\caption{case 3: the group is collectively $y$-monotone}
\label{ucase2}
\end{figure}

In the last case, there are two \Ushape 's grouped with the same third polyomino. In order to touch a \Ushape 's pocket, a polyomino must have a square in which three edges are exposed. Since any pentomino can have at most two of these partially exposed squares oriented in the $y$-direction, the largest number of \Ushape 's that be attached to any pentomino is two. After these three pentominoes are grouped together, no more \Ushape 's can attach, because the group no longer has any vertically oriented partially exposed squares. We further claim that any group formed by two \Ushape 's and a third pentomino must be $y$-monotone. The accuracy of these claims is verified by the case analysis shown in Figure~\ref{case3}.
\end{proof}

\begin{figure}[h]
\begin{center}
\includegraphics[height=1in]{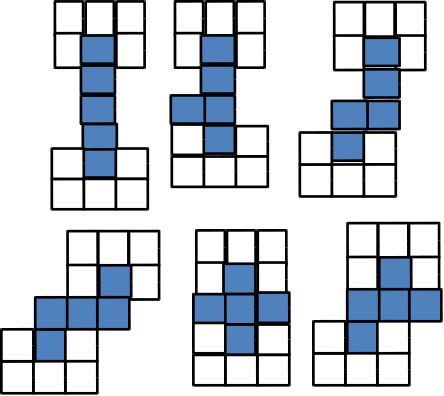}
\end{center}
\caption{case 4: the group is collectively $y$-monotone}
\label{case3}
\end{figure}

Thus, we have proved the claim that pentominoes cannot interlock. Next, we find a lower bound for polyomino interlockedness. We show that hexominoes can interlock. We also investigate the computational complexity of deciding interlockedness of systems containing hexominoes. 

\section{Interlockedness of Hexominoes}
Thus far, we have proved that collective interlockedness cannot occur for polyominoes with five or fewer squares. There are relatively simple nontrivial examples of interlocking octominoes. We show in this paper that a system of polyominoes containing hexominoes can also be interlocked. We present an example of interlocking hexominoes in Figure~\ref{fig:locked6s}. Before we prove the larger theorem, we present two important lemmas that will be useful in proving this theorem. The inspiration for these theorems comes from rules one and two in \cite{Connely}, which restrict rotation when lines are tightly confined within acute triangles.

\begin{lemma}\label{rectlemma1}
In any $\epsilon$-perturbation of the configuration shown in Figure~\ref{fignewrect} lying between the fixed lines $l$ and $m$, rectangle $ABCD$ has only been displaced horizontally for sufficiently small $\epsilon$. Or, equivalently, in any continuous motion of this configuration, this rectangle is pinned to horizontal sliders for positive time.
\end{lemma}

\begin{figure}[h]
\begin{center}
\includegraphics[height=1in]{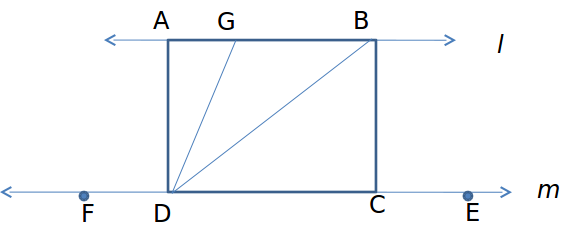}
\end{center}
\caption{A rectangle between two fixed lines}
\label{fignewrect}
\end{figure}

\begin{proof}
Let $A'B'C'D'$ be an $\epsilon$-displacement of rectangle $ABCD$. Let $\beta$ be the rotation angle from $ABCD$ to $A'B'C'D'$. By constraining $\epsilon$, we may constrain $\beta$ so that it is less than $\angle BDA$. By symmetry, we may constrain $\beta$ to being greater than or equal to zero, because a negative rotation angle is equivalent to a positive rotation angle about another point, since this rectangle has reflective symmetry. Given these constraints on $\beta$, we claim that if rectangle $A'B'C'D'$ is a rotation of $ABCD$ by $\beta$, the vertical distance between $B'$ and $D'$ is greater than the height of rectangle $ABCD$ (i.e. the vertical distance between $B$ and $D$). Since $\beta$ is less than $\angle BDA$, and $\angle BDE$ is $90^\circ{}- \angle BDA$, we know that $\angle BDE$ is greater than $90^\circ-\beta$. Therefore, since $\angle B'D'E$ is $\beta$ more than $\angle BDE$, $\angle B'D'E$ is closer to (but still under) $90^\circ$, and thus the vertical distance between $B'$ and $D'$ is greater than that between $B$ and $D$ for non zero $\beta$. 
Thus, rectangle $A'B'C'D'$ will intersect either lines $l$ or $m$, so the angle $\beta$ must equal $0$. Therefore, $A'B'C'D'$ must be strictly a horizontal or vertical translation of $ABCD$, and since the rectangle cannot intersect the lines $l$ or $m$, this translation must be strictly in the horizontal direction.
\end{proof}

\begin{lemma}\label{rectlemma2}
In any configuration $\epsilon$-close to Figure~\ref{fig:rect2} contained between fixed lines $l$ and $p$, each piece has simply been horizontally translated from its original position. In other words, each piece is pinned on horizontal sliders for positive time. In particular, no piece has been rotated. 
\end{lemma}

\begin{figure}[h] 
\begin{center}
\includegraphics[height=2.5in]{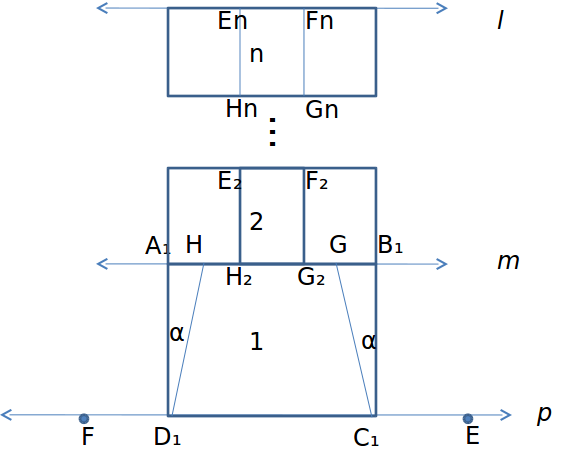}
\caption{$n$ rectangles between two fixed lines}
\label{fig:rect2}
\end{center}
\end{figure}

\begin{proof}
We divide segment $A_1B_1$ into fifths. We define points $G$ and $H$ on shape 1 such that $A_1H$ and $B_1G$ are both $\frac{A_1B_1}{5}$. We define the points $G_2$ and $H_2$ on shape 2 such that $GG_2$ and $HH_2$ are $\frac{A_1B_1}{5}$. We define points $E_2$ and $F_2$ on segment $A_2B_2$ such that $E_2F_2G_2H_2$ is a rectangle. Furthermore, we define points $E_iF_i$ and $G_iH_i$ on segments $A_iB_i$ and $C_iD_i$ on the $i$th rectangle such that $E_iF_iG_2H_2$ and $G_iH_iG_2H_2$ are rectangles. Let $\epsilon < \frac{A_1B_1}{10}$ be a sufficiently small constant. We define the angle $\alpha = \angle HD_1A_1 = \angle B_1C_1G$.  
Consider a configuration $\epsilon$-close to the configuration in Figure~\ref{fig:rect2} in which object $\chi$ has been moved to object $\chi '$; for example, point $A_1$ has been moved to point $A_1'$, shape 1 has been moved to shape 1', and similarly for all other points and shapes. In the configuration shown in Figure~\ref{fig:rect2}, we know that all the shapes are trapped between lines $l$ and $p$. 

Let $\beta$ be the rotation angle from $A_1B_1C_1D_1$ to $A_1'B_1'C_1'D_1'$. If $\epsilon$ is chosen small enough, $\beta$ is no larger than $\alpha$ in absolute value. By an argument similar to that used in Lemma~\ref{rectlemma1}, we can show that segment $G_2H_2$ is above line $m$. Since point $G_2$ is a distance of $2\epsilon$ away from $G_1$ and point $H_2$ is a distance of $2\epsilon$ from point $H_1$, we know that the segment $G_2H_2$ must stay above the segment $G_1H_1$, because the shapes cannot translate so that point $G_2$ protrudes past point $G_1$, and vice versa. Thus, rectangle $E_2'F_2'G_2'H_2'$ is trapped strictly above line $m$, and the region of intersection of the first $n-1$ rectangles with region $E_nF_nG_2H_2$ is a rectangular region within each of the first $n-1$ rectangles with width $G_2H_2$ that is trapped between the lines $l$ and $m$. 

We can apply this argument inductively to prove that every shape trapped between the lines $l$ and $m$ has a smaller rectangle inside it that is bounded between two other lines, as long as $\epsilon$ is constrained to $\frac{1}{10}$th the width of the smallest rectangular region being trapped. By Lemma~\ref{rectlemma1}, these smaller rectangular regions can only have been translated horizontally. Since these regions cannot move independently of the larger rectangles, all of the original rectangles can only have been translated horizontally, given that they have only been displaced by small enough $\epsilon$.
Similarly, there must be a rectangular region within shape 1 that must stay below shape 2 and thus below line $m$. Thus, there is a rectangle within shape 1 that has only been translated horizontally, by Lemma~\ref{rectlemma1}. Overall, rectangle 1 must have been translated horizontally, since it cannot move independently of its inner trapped rectangle.  
\end{proof}

\begin{figure}
\begin{center}
\includegraphics[height=3in]{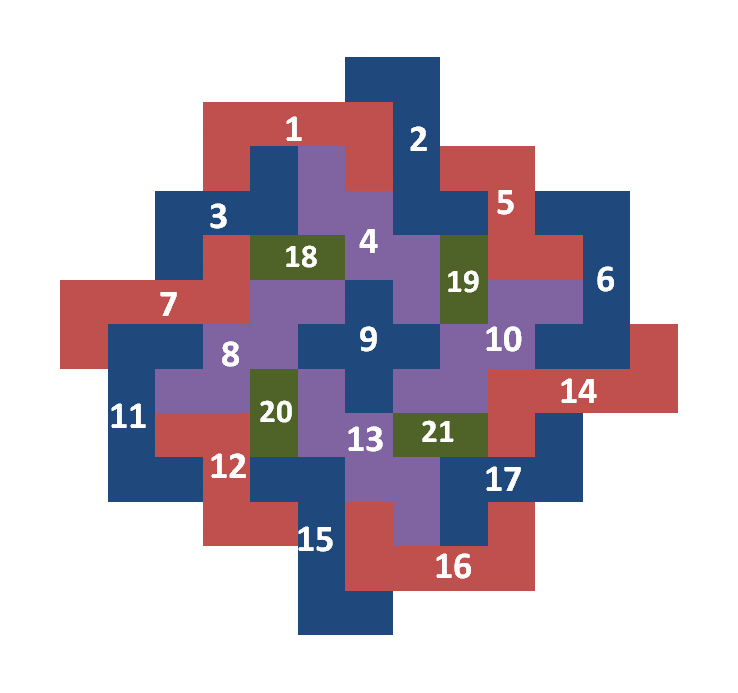}
\end{center}
\caption{Interlocking Hexominoes}
\label{fig:locked6s}
\end{figure}

\begin{theorem}
The system of polyominoes shown in Figure~\ref{fig:locked6s} is rigid, and therefore interlocked.
\end{theorem}

\begin{proof}

Roughly, we will first show that none of the shapes can rotate. Next, we will use the definition of rigidity to show that if one shape is fixed in the plane, all shapes are fixed.

Suppose shape 1 is fixed in the plane. As we can see, small portions of shapes 3 and 4 are trapped within shape 1. Therefore, these shapes are pinned to vertical sliders for positive time, by Lemma~\ref{rectlemma2}. As we can see, parts of shapes 7 and 18 are between shapes 3 and 4. Since the trapped portions of shapes 7 and 18 have width 1, we know that they must be between shapes 3 and 4 if $\epsilon < \frac{1}{2}$, because they will not be able to slide past each other in the vertical direction. Thus, shapes 3 and 4 impose the constraints of Lemma~\ref{rectlemma2} on shapes 7 and 18 for small enough $\epsilon$. Thus, shapes 7 and 18 are fixed on vertical sliders. We can use the same argument to show that shapes 8 and 9 are pinned to horizontal sliders by shapes 7 and 4. Furthermore, shape 11 is pinned to horizontal sliders by shapes 7 and 8. Thus, shape 11 is fixed on vertical sliders relative to shape 1 so it can not rotate. By symmetry, we claim that no polyomino in this configuration can rotate.

Now, we prove that this configuration is rigid. In other words, given that shape 1 is fixed, any continuous deformation of this configuration is constant for positive time. Recall that shapes 7, 8, 9, and 11 are fixed on vertical sliders (relative to 1). This implies that 20 and 13 are on vertical sliders by Lemma~\ref{rectlemma2}, because they are strictly between shapes 8 and 9. Since 12 is between 20 and 11, it is also on a vertical slider relative to 11. Also, notice that shapes 8 and 12 are fixed on horizontal sliders relative to 11. Since they are all on both horizontal and vertical sliders relative to each other, shapes 8, 11, and 12 cannot be displaced relative to each other, and so are fused together for positive time in any motion. By symmetry, shapes 1, 3, and 4, shapes 5, 6, and 10, shapes 11, 8, and 12, and shapes 13, 16, and 17 are likewise fused.

\begin{figure}
\begin{center}
\includegraphics[height=3in]{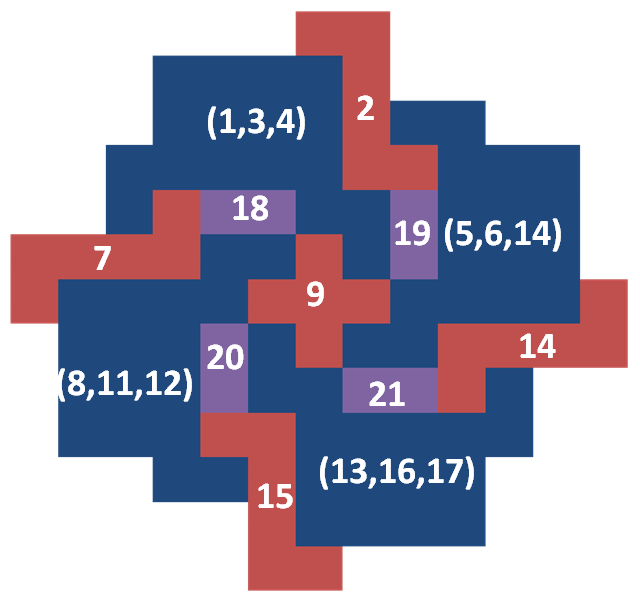}
\end{center}
\caption{Interlocking hexominoes, with fused shapes indicated}
\label{fig:locked62}
\end{figure}

As we can see, shape 20 is pinned to vertical sliders relative to 1 by shape (8, 11, 12) and shape 9 along with shape (13,16,17). Shape 15 is also pinned to vertical sliders between shapes 13 and (8, 11, 12). Shapes 20 and 15 are also fixed to horizontal sliders relative to (8, 11, 12). Thus, by symmetry, shapes 1, 3, 4, 18, and 7, shapes 11, 8, 12, 20, and 15, shapes 16, 13, 17, 21, 14, and shapes 6, 5, 10, 19, and 2 are all fused.

Shape (13, 16, 17, 14, 21) is trapped on vertical sliders relative to shape (11, 8, 12, 15, 20) and shape 9. Shapes (13, 16, 17, 14, 21) and 9 are fixed to horizontal sliders relative to (11, 8, 12, 15, 20). Therefore, shapes (13, 16, 17, 14, 21), (11, 8, 12, 15, 20), and 9 are all fused. By symmetry, all shapes are fixed.
\end{proof}

\section{PSPACE Hardness of Interlocking Hexominoes}

It follows from the work of Demaine and Hearn \cite{Hearn} that the problem of deciding polyomino interlockedness is PSPACE Hard. We now outline this argument. In \cite{Hearn}, it was shown that the Totally Quantified Boolean Formula (TQBF) problem, which is PSPACE complete, can be reduced to the sliding blocks problem. The goal of the TQBF problem is to find an assignment of $1$ or $0$ to variables of a given boolean formula (subject to certain additional constraints on the variables) such that the formula is satisfied.  Hearn and Demaine reduce the Boolean AND and OR operations necessary to realize TQBF to the two sliding block gadgets shown in Figure~\ref{fig:dominogates}. 

\begin{figure}[h]
\begin{center}
\includegraphics[height=1.5in]{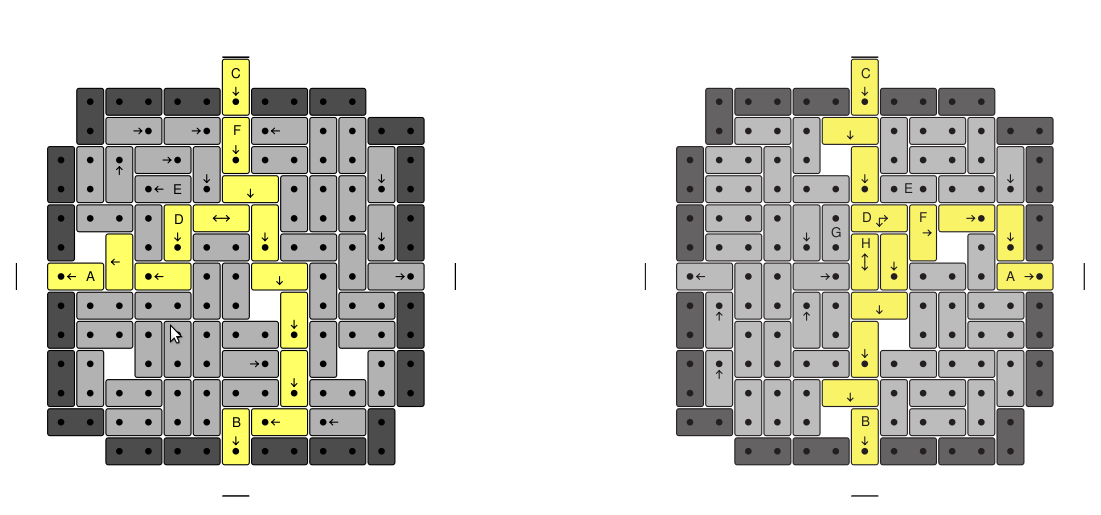}
\caption{Gates realized with dominoes. Taken from Demaine and Hearn \cite{Hearn}.}
\label{fig:dominogates}
\end{center}
\end{figure}

Although their gadgets use only dominoes, an implicit requirement in their problem is a large, rigid boundary that will hold all of their gadgets together. In a different paper, Demaine et al. \cite{Uehara} created a construction involving a boundary that is only non-interlocking if one key piece is moved backwards by one square. If we combine this boundary with their AND and OR constructions, we can rephrase the sliding block representation of TQBF as a representation of polyomino interlockedness. We know that deciding whether or not a key peice can be moved back one square in the sliding block puzzle is PSPACE hard. Since it is possible to create a system of large polyominos whose interlockedness is contingent upon the translation of a key peice, it is possible to reduce sliding blocks to polyomino interlockedness. Thus, the problem of deciding polyomino interlockedness is PSPACE hard.

Here, we show how to construct a boundary that disassembles after the translation of one key peice using hexominoes (Figure~\ref{fig:pspace6s}). This, paired with the argument in the above reduction, implies that the problem of deciding the interlockedness of system of shapes containing only hexominoes and smaller polyominoes is PSPACE hard. We prove the boundary is rigid after discussing a key lemma.

\begin{theorem}
Determining interlockedness is PSPACE hard for a system of polyominoes with only Hexominoes and smaller polyominoes.
\end{theorem}
 
\begin{figure}
\begin{center}
\includegraphics[height=3.5in]{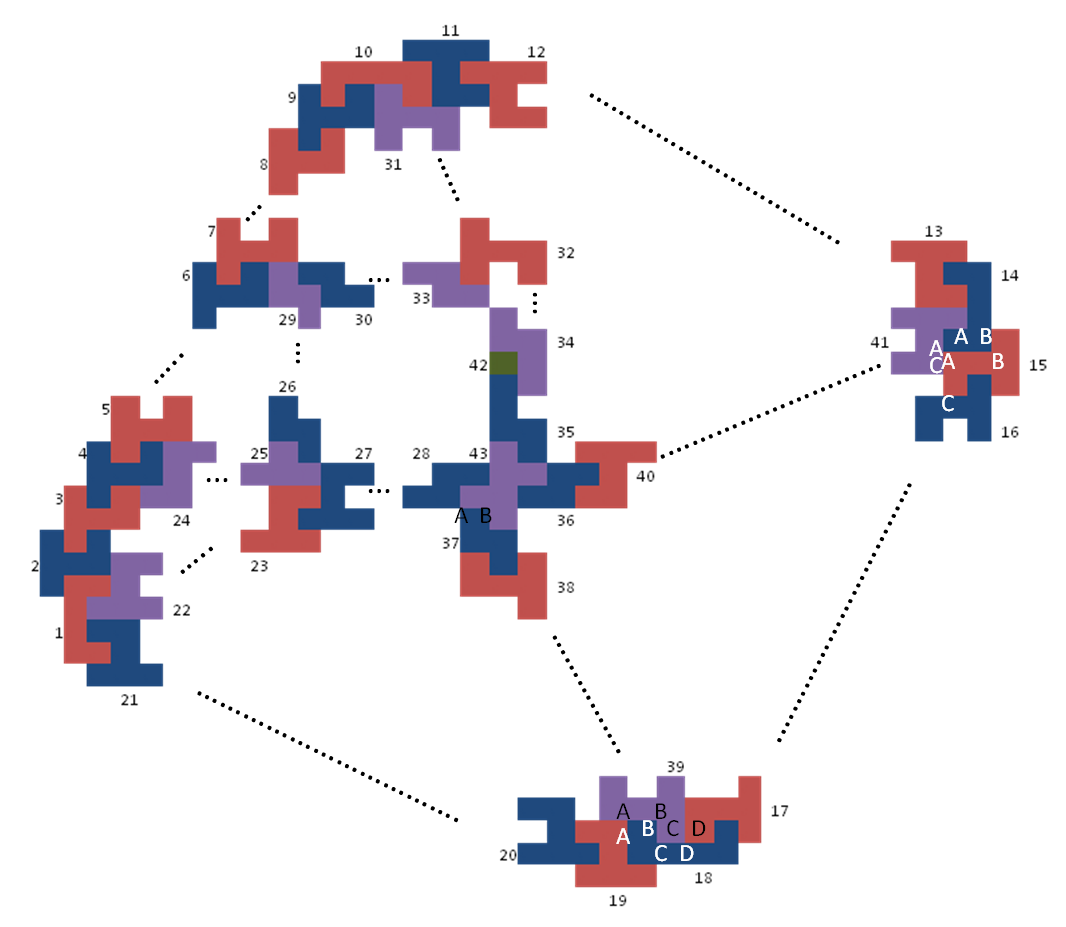}
\end{center}
\caption{ PSPACE hardness of interlocking hexominoes}
\label{fig:pspace6s}
\end{figure}

\begin{proof}
Here, we present a configuration of polyominoes with six or fewer squares in Figure~\ref{fig:pspace6s}. We prove that the configuration of hexominoes in this figure simulates a boundary that can disassemble if and only if one key piece is moved backwards by one square. As discussed previously, we can fill this boundary with dominoes so that deciding if the key piece can be moved backwards is PSPACE hard. We begin the proof by stating an important lemma.

\begin{lemma}\label{Zlemma}
In a continuous deformation of the configuration shown in Figure~\ref{fig:pspace6lemma}, each shape can slide right if and only if shape 1 can slide right, and each shape can slide left if and only if shape 4 can slide left.
\end{lemma}

\begin{figure}
\begin{center}
\includegraphics[height=2in]{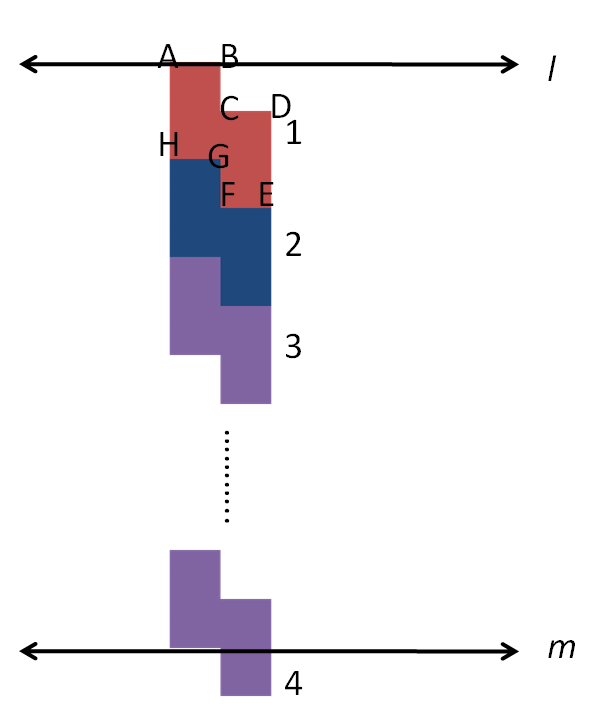}
\end{center}
\caption{PSPACE hardness of interlocking hexominoes}
\label{fig:pspace6lemma}
\end{figure}

\begin{proof}
By Lemma~\ref{rectlemma2}, all of these shapes are pinned to horizontal sliders for positive time. In other words, each piece can slide either to the left or the right, and cannot rotate. As we can see in Figure~\ref{fig:pspace6lemma}, shape 2 cannot slide right if shape 1 cannot slide right, because this would cause line $B_2C_2$ to cross line $F_1G_1$. We can apply the same argument to show that shape 3 cannot slide right unless shape 2 slides right. Thus, shape 3 can only slide right if shape 1 slides right. Applying this argument recursively, we show that all shapes between 1 and 4 cannot slide right unless shape 1 slides right. By symmetry, we show that all the shapes between 1 and 4 cannot slide left unless shape 4 slides left. 
\end{proof}
 
We now show that the configuration in Figure~\ref{fig:pspace6s} is rigid unless shape 42 is moved one square to the left. Recall that if a configuration is rigid, then any $\epsilon$-perturbation of the configuration has not changed, given that one shape is fixed and that $\epsilon$ is sufficiently small. We assume that shape 1 is fixed. Given small enough displacement $\epsilon$, it follows from recursive application of Lemma~\ref{rectlemma2} that shapes 22--23 and 27-- 28 are pinned to horizontal sliders (relative to 1), and that shapes 19-21 are pinned to horizontal sliders (relative to 1).
By the same argument, shapes 37--39 and shapes 16--17 are pinned to vertical sliders relative to shape 18. In other words, these shapes cannot translate horizontally or rotate relative to shape 18. 

We now show that the distance between points $A_{18}$ and $A_{37}$ cannot decrease. Recall that shape 39 cannot rotate or slide horizontally relative to shape 18. It can only slide vertically relative to shape 18. The distance between corresponding points on segments $A_{39}B_{39}$ and $A_{18}B_{18}$, as well as the distance between segments $C_{18}D_{18}$ and $A_{39}B_{39}$ can only increase. We can apply the same argument to show that the distance between each of the shapes between 37 and 39 can only have increased or stayed the same. It follows from repeated application of this argument that the overall distance between points $A_{18}$ and $A_{37}$ cannot decrease. 

We now show that shapes 38--39 and 18 are fused (i.e. cannot move relative to each other for positive time given a continuous deformation). Recall that shapes 28 and 19 are pinned to horizontal sliders. This means that shape 18 cannot slide down, disregarding collective rotation. Since shape 39 must stay above shape 18, it also cannot slide down. It follows by applying this argument repeatedly that all shapes between 37 and 39 cannot slide down. We know that shape 43 cannot translate upwards. Furthermore, we observe that if shape 43 has a positive rotation angle $\beta$, the vertical distance between points $D_{43}$ and $A_{43}$ will increase. If shape 43 has a negative rotation angle, the vertical distance between points $C_{43}$ and $D_{43}$ will increase. It follows that some portion of shape 43' must be lower than shape 43 given any nonzero rotation angle (where 43' is 43 after it is moved). Given that shapes 18 and 37--39 cannot collectively rotate, it follows that shape 37 cannot slide up, regardless of the movements of shape 43. This implies that shapes 37--39 and 18 cannot slide up. Since shapes 37--39 and 18 are pinned to vertical sliders relative to each other, and since they cannot slide apart from each other vertically, they are fused together for positive time. We can apply the same argument to show that shapes (22--23 and 27--28) are fixed, and that shapes (31--32 and 10) and shapes (36, 40--41, and 14) are fused.

We now show that collective shapes cannot rotate. If shape (37--39 and 18) had a positive rotation angle, the vertical distance between all points on segments $A_{37}B_{37}$ and $D_{18}E_{18}$ would increase. We know that shape 19 cannot rotate or translate downwards, as it is pinned to horizontal sliders (relative to 1). Therefore, all points on segment $E_{18}F_{18}$ cannot move down ($F_{18}$ is $2\epsilon$ away from $F_{19}$). This means that all points on segment $A_{37}B_{37}$ must move up. Recall, however, that given small enough $\epsilon$, some portion of shape 43' must be below shape 43. It follows that shape (37--39 and 18) cannot have a positive rotation angle, or else 43' and 37' will intersect. By the same argument, shapes (31--32 and 10) and shapes (36, 40--41, and 14) cannot have a positive rotation angle. We now show that this shape cannot have a negative rotation angle. Recall that shapes 15--17 are pinned to horizontal sliders relative to shape 18. Thus, they cannot rotate independently of shape 18. If shapes 15--17 and (37--39 and 18) had a negative rotation angle, shape 15' would be to the left of shape 15. Since segment $A_{15}C_{15}$ is to the right of segment $A_{41}C_{41}$, each point on segment $A_{41}'C_{41}'$ must be on the left of segment $A_{41}C_{41}$. By Lemma~\ref{Zlemma}, shape (36, 40--41, and 14) cannot translate right. However, this shape can have a negative rotation angle. Thus, we have shown that (37--39 and 18) had a negative rotation angle only if (36, 40--41, and 14) had a negative rotation rotation angle. In order for (36, 40--41, and 14) to have a negative rotation angle, shape (31--32 and 10) must, by symmetry, slide right. However, shape (31--32 and 10) cannot slide right, because the shapes between 32 and 43 cannot slide together. Furthermore, shape (31--32 and 10) and shapes 2--10 cannot have a negative rotation angle, because by symmetry, this would require shape 1 to slide right. However, shape 1 is fixed.

\begin{figure}
\begin{center}
\includegraphics[height=2in]{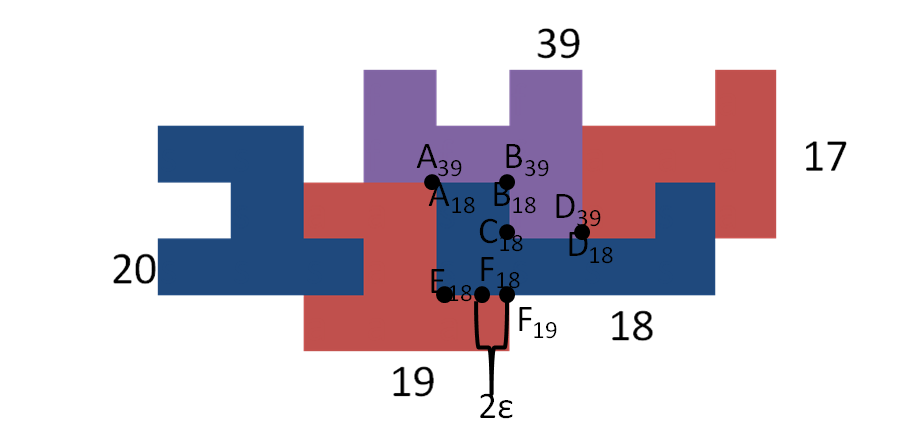}
\end{center}
\caption{PSPACE hardness of interlocking hexominoes}
\label{fig:moredetail}
\end{figure} 

We now show that shape (37--39 and 18) are fused with shapes 15--17. We know that collective rotation is impossible. We know that shapes 15--17 are pinned to vertical sliders relative to shape 18. We also know that none of these shapes can translate downwards relative to each other without intersection. As shown in Figure~\ref{fig:pspace6s}, shape 15 cannot slide up unless segment $A_{14}'B_{14}'$  is above segment $A_{14}B_{14}$ . As mentioned previously, shape (36, 40--41, and 14) cannot slide up, and cannot have a negative angle of rotation. Therefore, shape 15 cannot slide up. It follows that shapes 16--17 cannot slide up. Conversely, shape 17 cannot slide down without intersecting shape 18. It follows that shapes 15--17 cannot slide down. Since shapes 15--17 are trapped on vertical sliders and cannot slide up or down, they are all fused to shape (37--39 and 18). By symmetry, shapes (22--23, 27--28, and 19--21) are fixed, and shapes (31--32 and 2--10) and shapes (36, 40--41, and 11--14) are fused.  

We now show that all shapes outside the boundary are fixed. Shape (15--18 and 37--39) cannot slide up without intersecting shape 43. It cannot slide down without intersecting 19. It cannot slide left without intersecting shape 19. It cannot move right unless, by Lemma~\ref{Zlemma}, the shapes between 35 and 32, and shape (31--32 and 2--10) slides right. However, shape (31--32 and 2--10) cannot slide right without intersecting shape 1. Thus, shape (15--18 and 37--39) is fixed. By symmetry, shape  (36, 40--41, and 11--14) is also fixed. Shape (31--32 and 2--10) cannot slide right without intersecting shape 1. It cannot slide down without intersecting shape 1, and it cannot slide up without intersecting shape 11. It can slide left only if the shapes between 32 and 43 can slide left, and shape 43 cannot slide left without intersecting shape 37. Thus, shape (31--32 and 2--10) is fixed. 

We now show the boundary is fixed. As we can see, shapes 24--25 are pinned to vertical sliders between shapes 4 and 27 by Lemma~\ref{rectlemma2}. Shape 24 cannot slide up without intersecting shape 5. Shape 25 cannot slide down without intersecting shape 23. Thus, by Lemma~\ref{Zlemma}, these shapes are fixed. Similarly, the shapes between 26 and 29 are pinned to horizontal sliders by 23 and 7. They can only slide right if shape 29 can slide right. However, shape 29 is pinned to vertical sliders by shapes 6 and 32. Thus, the shapes between 29 and 26 are fixed. The shapes between 29 and 33 can only move if shape 33 can slide up. However, shape 33 cannot slide up without intersecting shape 32. Thus, the shapes between 29 and 33 are fixed. The shapes between 32 and 35 are pinned to horizontal sliders. They can only move if shape 35 can slide left. However, shape 35 cannot slide left without intersecting shape 43. Thus, the shapes between 32 and 35 are fixed. Finally, we show that shape 43 is fixed. Shape 43 cannot slide left, right, up, or down by Lemma~\ref{Zlemma}. Recall that if shape 43 had a nonzero rotation angle, some portion of segment $A_{43}'B_{43}'$ would be below $A_{43}B_{43}$. Since shape 37 is fixed, this would cause shape 43 to intersect shape 37. Thus, all shapes are fixed, given that shape 42 cannot slide left. If shape 42 slides left one unit, shape 35 can slide up one unit. This will allow shapes (15--18 and 37--39), 43, and  (36, 40--41, and 11--14) to slide right. It can be checked that the rest of the boundary will disassemble.

\end{proof}

\section{Conclusion}
In conclusion, the question that was posed was answered completely in this research. We proved that a system of polyominoes with five or fewer squares cannot interlock by presenting an algorithm to separate them. We recreated Hearn's and Demaine's proof of PSPACE hardness of interlocking polyominoes using only hexominoes and smaller polyominoes. 
As has been discussed in the work of Erik Demaine [7] and [5], the problem of determining interlockedness is related to path planning problems, which are important and difficult problems in Artificial Intelligence. Many of the structures discussed in this paper are extremely rigid and therefore impossible to break apart without damaging the individual shapes' internal structures. Thus, interlockedness may be a potentially useful element in the evolution of the internal structures of extremely rigid proteins, including, perhaps, the proteins in the cell wall. Whether or not such a model is useful should be investigated further.  
Further interesting research on this topic includes discovering further subsets of polyominoes that can and cannot interlock, and generalizing these results to polyominoes (and perhaps other shapes) in three dimensions, which may be much more complex and important.

\section{Acknowledgements}
I would like to thank my mentor, Zachary Abel, who is a graduate student at MIT, and who oversaw this project. I would like to thank CEE (Center for Excellence in Education) and RSI (Research Science Institute) for giving me the opportunity to do this research. I would like to thank MIT for its support of the RSI program. I would like to thank Mr. John Yochelson and Dr. Christine Hill for sponsoring my stay at RSI. I would like to thank Dr. Tanya Khovanova, Dr. Kartik Venkatram, and Professor David Jerison of the MIT mathematics department for giving me advice about the direction of this project. I would like to thank the RSI staff, especially Dr. John Rickert, for assisting in the writing of this paper.

%% file: biblio.tex
% This is biblio.tex
 
%% Enter your formatted bibliography information into this file.
%% Use a format appropriate to a journal in your field.

%% DO NOT TOUCH THIS LINE UNLESS YOU REALLY KNOW WHAT YOU'RE DOING